\documentclass[12pt,a4paper]{amsart}
\usepackage{graphicx}
\usepackage{color}
\usepackage{enumerate}
\usepackage{amssymb}
\usepackage[colorlinks=true]{hyperref}
\hypersetup{urlcolor=blue, citecolor=blue, draft=false}
\newtheorem{theorem}{Theorem}[section]
\newtheorem{lemma}{Lemma}[section]

\theoremstyle{remark}

\theoremstyle{remark}

\usepackage{subfigure}
\theoremstyle{definition}

\begin{document}
\title[On certain analytic functions defined by differential inequality]{On certain analytic functions defined by differential inequality}

\author[Prachi  Prajna Dash, Jugal  Kishore  Prajapat]{Prachi Prajna Dash$^\ddagger$, Jugal Kishore Prajapat$^\dagger$}

\address{$^\ddagger$Department of Mathematics, Central University of Rajasthan, Bandarsindri, Kishangarh-305817, Dist.- Ajmer, Rajasthan, India}
\email{prachiprajnadash@gmail.com, jkprajapat@gmail.com}

\date{}

\begin{abstract}
For the family of analytic functions $f(z)$ in the open unit disk \(\mathbb{D}\) with \(f(0)=f'(0)-1=0\), satisfying the differential equation
\begin{equation*}
zf'(z) - f(z) = \dfrac{1}{2} z^2 \phi(z), \quad |\phi(z)| \leq 1,
\end{equation*}
we obtain radii of convexity, starlikeness, and close-to-convexity of partial sums of $f(z)$. We also study the generalization of this family having the form
\begin{equation*}
zf'(z) -f(z) = \lambda z^2 \phi(z), \quad |\phi(z)| \leq 1,
\end{equation*}
where $\lambda > 0,$ and obtain some useful properties of these functions.
\end{abstract}


\maketitle
%-------------------------------------------
\section{Introduction}
\setcounter{equation}{0}

Let $\mathcal{A}$ denote the family of analytic functions in the open unit disk $\mathbb{D}=\{z\in\mathbb{C}: |z|<1\}$ satisfying $ f(0)=0=f'(0)-1$, and let $\mathcal{S}$ be the subclass of univalent functions in $\mathcal{A}$. Let $E\subset \mathcal{A}$ and $P$ denote the property of the image of $f(z) \in E$. The largest $\gamma_f$ such that $f(z) \in E$ has a property $P$ in each disk $\mathbb{D}_{r}= \{z \in \mathbb{C}: \, |z|<r\}$ for $0<r\leq \gamma_f$ is known as the radius of the property $P$ of $f(z)$. The number $\gamma := \mbox{inf}\{\gamma_f: \; f (z)\in E\}$ is the radius of the property $P$ of the class $E$. If $\mathcal{G}$ is the class of all $f(z) \in E$ characterized by the property $P$, then $\gamma$ is called the $\mathcal{G}$-radius of the class $E$.

A function $f(z) \in \mathcal{A}$ is said to be convex if it maps $\mathbb{D}$ conformally onto a convex domain. A function $f(z)\in \mathcal{A}$ is convex if and only if $\Re \left(1+zf''(z)/f'(z)\right)>0$ for all $z$ in $\mathbb{D}$. The class of these functions is denoted by $\mathcal{K}.$ A function $f(z) \in \mathcal{A}$ is said to be starlike if it maps $\mathbb{D}$ conformally onto a starlike domain. A function $f(z)\in \mathcal{A}$ is starlike if and only if $\Re \left(zf'(z)/f(z)\right)>0$ for all $z$ in $\mathbb{D}$. The class of these functions is denoted by $\mathcal{S}^*$. Note that $ \mathcal{K} \subset \mathcal{S}^* \subset \mathcal{S}$. An analytic function $f(z) \in \mathcal{A}$ is close-to-convex in $\mathbb{D}$ if $\mathbb{C} \backslash f(\mathbb{D})$ can be represented as a union of non-crossing half lines. The class of these functions is denoted by $\mathcal{C}$. For every convex function $g(z)$, a function $f(z) \in \mathcal{C}$ satisfies the inequality $\Re \left(f'(z)/g'(z)\right)>0$ for $z$ in $\mathbb{D}$. 

A function $f(z)$ is subordinate to the function $g$, written as $f(z) \prec g(z)$, if there exists a Schwarz function $w(z)$, with $w(0)=0$ and $|w(z)| \leq |z|$, such that $f(z)=g(w(z))$ for all $z \in \mathbb{D}$. If $g(z) \in \mathcal{S}$, then the following relation:
\[f(z) \prec g(z) \quad \Longleftrightarrow \quad [f(0) = g(0) \;\;\mbox{and} \;\; f(\mathbb{D}) \subset g(\mathbb{D})],\]
holds true. The convolution (or Hadamard product) of two analytic functions $f(z)=\sum_{k=0}^{\infty}  a_k z^k $ and $g(z)=\sum_{k=0}^{\infty}  b_k z^k $ is defined by
    \[(f*g)(z)=\sum_{k=0}^{\infty}  a_k \, b_k \,z^k.\]
The convolution has the algebraic properties of ordinary multiplication. Rusheweyh and Sheil-Small \cite{rusch} proved an important conjecture known as the {\it P\'{o}lya-Schoenberg Conjecture} that the class $\mathcal{K}$ is preserved under convolution. They also proved that the classes $\mathcal{S^*}$ and $\mathcal{C}$ are closed under convolution with the class $\mathcal{K}$. Many other convolution problems were studied by Rusheweyh in \cite{rusch2} and have found many applications in various fields. Finding classes of univalent functions that are preserved under convolution and studying their properties is an important research problem. Recently, Peng and Zhong \cite{peng} proved that $f(z) \in \mathcal{A}$ that satisfy the differential equation 
\begin{equation}\label{omega}
zf'(z) - f(z) = \dfrac{1}{2} z^2 \phi(z), \quad |\phi(z)| \leq 1,
\end{equation}
are closed under convolution. Let $\Omega$ be the subset of $\mathcal{A}$ that consists of all functions satisfying \eqref{omega}. Note that the functions $f\in \Omega$ satisfy the inequality 
\begin{equation}\label{omega1}
|zf'(z) -f(z)| <\dfrac{1}{2}, \qquad z \in \mathbb{D}.
\end{equation}

The functions in $\Omega$ are univalent and starlike in $\mathbb{D}$ (see \cite{peng, hesam}). Obradovi\'{c} and Peng \cite{obradovic} gave certain sufficient conditions for functions belonging to the class $\Omega$. Also, Wani and Swaminathan \cite{wani} obtained certain sufficient conditions and subordination properties for functions belonging to the class $\Omega$. 

To obtain an example of a function in $\Omega$, let $f_{\mu}(z) \in \mathcal{A}$ such that
\begin{equation*}
    zf'_{\mu}(z) - f_{\mu}(z) = \dfrac{1}{2} z^2  \phi_{\mu}(z), 
\end{equation*}
where $\phi_{\mu}(z)=\dfrac{\mu+z}{1+\mu z}, \; \mu \in [-1, 1]$. Then
\begin{equation*}
    zf'_{\mu}(z) - f_{\mu}(z) = \dfrac{1}{2} z^2 \left(\mu+(1-\mu^2)\sum_{k=1}^{\infty} (-\mu)^{k-1} \,z^k \right), 
\end{equation*}
now comparing the coefficients on both sides, we obtain that $f_{\mu}(z) \in \Omega$, where
\begin{equation}\label{example}
 f_{\mu}(z)=z + \dfrac{\mu}{2}z^2 + \dfrac{1-\mu^2}{2}\sum_{k=1}^{\infty} \dfrac{(-\mu)^{k-1}}{k+1} \,z^{k+2}, \quad \mu \in [-1, 1].   
\end{equation} 

For $\lambda > 0$, let $\Omega_{\lambda}$ denote the set of functions in $\mathcal{A}$ satisfying the differential equation
\begin{equation}\label{oa1}
zf'(z) -f(z) = \lambda z^2 \phi(z), \quad |\phi(z)| \leq 1.
\end{equation}
Note that the functions $f(z)\in \Omega_{\lambda}$ satisfy the inequality
\begin{equation}\label{oa}
|zf'(z) -f(z)| <\lambda, \qquad z \in \mathbb{D}.
\end{equation}
The class $\Omega_{\lambda}$ is non-empty as the identity function satisfies the inequality \eqref{oa}. Note that $\Omega_{1/2}=\Omega$ and $\Omega_{\lambda_1} \subseteq \Omega_{\lambda_2}$ when $\lambda_1 \leq \lambda_2$. Consequently, the functions in $\Omega_{\lambda}$ are univalent and starlike in $\mathbb{D}$ for $\lambda \in \left(0, 1/2 \right]$. We observe that the function
\begin{equation} \label{egOA}
f(z)=z+\dfrac{\lambda z^2}{2}+ \dfrac{\lambda z^3}{4}, \qquad \lambda > 0, \, z \in \mathbb{D},
\end{equation}
belongs to the class $\Omega_{\lambda}$. 

A function $f(z)\in \mathcal{A}$ has the power series representation 
\begin{equation}\label{class A}
f(z)= z+\sum_{k=2}^\infty  a_k \,z^k, \qquad \left( a_k =\dfrac{f^{(k)}(z)}{k!},\;k=2,\,3,\dots,\; z \in \mathbb{D}\right),
\end{equation} 
and the $n$-th partial sum of $f(z)$ is the polynomial 
\begin{equation}\label{partial_sum}
s_n(z; f)=z+\sum_{k=2}^n  a_k \,z^k.
\end{equation} 
In 1928, Szeg\"{o} demonstrated that \( s_n(z; f) \) for \( f(z) \in \mathcal{S} \) are univalent in the disk \( \mathbb{D}_{1/4} \), and this radius cannot be extended to a larger value \cite{szego}. Later, in 1941, Robertson \cite{robertson} established that \( s_n(z; f) \) remains univalent in the disk \( |z| < 1 - (3 \log n)/n \) for all \( n \geq 5 \). It is true that the Koebe function is extremal in many subclasses of \( \mathcal{S} \), but Bshouty and Hergartner \cite{bshouty} showed in 1991 that it is not extremal for the radius of univalence of \( s_n(z; f) \) for \( f(z) \in \mathcal{S} \). This result was resolved using the convolution theorem by Ruscheweyh and others \cite{rusch} for various geometric subclasses of \( \mathcal{S} \). However, determining the exact largest radius of univalence of \( s_n(z; f) \) for \( f(z) \in \mathcal{S} \) remains an open problem (see \cite[page 246]{duren}).
%see page 406 of \cite{bshouty}, which showed that the Koebe function is not extremal for some $f\in\mathcal{S}$.

\medskip
We shall use the following results. 
\begin{lemma}\label{lemma1} {\cite{peng, hesam}} 
If $f(z)=z+\sum_{k=2}^{\infty}  a_k z^k \in \Omega$, then 
\begin{enumerate}
\item[(a)] \;\;$|z|-\dfrac{|z|^2}{2} \leq |f(z)| \leq |z|+\dfrac{|z|^2}{2}$;
\item[(b)] \;\;$1-|z| \,\leq \,|f'(z)| \,\leq \,1+|z|$;
\item[(c)] \;\;$| a_k| \,\leq \,\dfrac{1}{2(k-1)}, \;k\geq 2$.
\end{enumerate}
For each $z \in \mathbb{D}\setminus \{0\},$ equality occurs in (a) and (b) if and only if $f(z)=z+\dfrac{1}{2} e^{i \psi} z^2,$ where $\psi \in \mathbb{R}$, and equality holds for (c) if $f(z)=z+\dfrac{1}{2(k-1)} z^k$, where $k=2,3,\dots$.
\end{lemma}

%\vspace{0.2in}
\begin{lemma}\label{lemma3}{\cite{wani}}
If $f(z) \in \Omega$, then
\[ \bigg| \dfrac{zf'(z)}{f(z)}-1 \bigg| \leq \dfrac{r}{2-r}, \quad |z|=r<1,\] 
for each $z \in \mathbb{D}$.
\end{lemma}

%\vspace{0.2in}
\begin{lemma}\cite{rogo}\label{l2}
    Let $q(z)=\sum_{n=1}^{\infty}b_nz^n\in \mathcal{K}$. If $p(z)=\sum_{n=1}^{\infty}a_nz^n$ is analytic in $\mathbb{D}$ and satisfies the subordination $p(z)\prec q(z)$, then $|a_n| \leq |b_1|$ where $n\in \mathbb{N}$.
\end{lemma}

%\vspace{0.2in}
%\begin{lemma}\label{lemma2a}
%Let $r\in [0,1)$ and $n \ge 2$. If
%\begin{align*}
%S_1(n,r)&=\sum_{k=1}^{\infty}\dfrac{r^{k+n}}{k+n-1},   \\
%S_2(n,r)&=\sum_{k=1}^{\infty}(k+n) \,r^{k+n-1},  \\ 
%S_3(n,r)&= \sum_{k=1}^{\infty} \dfrac{k+n}{k+n-1} \,r^{k+n},
%\end{align*}
%then $S_i(n,r)\le S_i(2,r)$ for $i=1,2,3$.
%\end{lemma}
%\begin{proof} Note that
%\begin{equation*}
%S_1(n+1,r)-S_1(n,r) =  \sum_{k=1}^{\infty} \dfrac{r^{k+n+1}}{k+n} - \sum_{k=1}^{\infty} \dfrac{r^{k+n}}{k+n-1} =-\dfrac{r^{n+1}}{n} \leq 0.
%\end{equation*}
%Hence $S_1(n,r)$ is decreasing for all $n \geq 2$. Similarly, we can show that $S_2(n,r)$ and $S_3(n,r)$ are decreasing.
%\end{proof}

%\vspace{0.2in}
\begin{lemma}\label{lemma3a}
Let $f(z)=z+\sum_{k=2}^{\infty}  a_k z^k \in \Omega$ and $\rho_n(z; f)=\sum_{k=n+1}^{\infty}  a_k z^k$. Then, for $n \ge 2$ and $|z|=r <1$, 
\begin{align}
|\rho_n(z ;f)| &\leq\,-\,\dfrac{r}{2}\,\mbox{ln}(1-r)-\dfrac{r^2}{2}, \label{ps1} \\ 
|\rho'_n(z; f)| &\leq \,\dfrac{2r^2-r}{2(1-r)}-\dfrac{\mbox{ln}(1-r)}{2}, \label{ps2} \\
|z \rho''_n(z; f)| &\leq \,\dfrac{r^2(3-2r)}{2(1-r)^2}.\label{ps3}
\end{align}
\end{lemma}

\begin{proof} Using Lemma \ref{lemma1}$(c)$, we obtain
\begin{align*}
|\rho_n(z; f)|  & \leq \; \dfrac{1}{2} \sum^{\infty}_{k=1}\dfrac{r^{k+n}}{k+n-1} \,\leq  \,\dfrac{1}{2} \sum ^{\infty}_{k=1} \dfrac{r^{k+2}}{k+1} = -\dfrac{r}{2} \ln(1-r) - \dfrac{r^2}{2}, \\
|\rho'_n(z; f)| & \leq \; \dfrac{1}{2} \sum_{k=1}^{\infty} \dfrac{(k+n)\,r^{k+n-1}}{k+n-1}  \leq \dfrac{1}{2} \left(\sum_{k=1}^{\infty} r^{k+1}+\sum_{k=1}^{\infty} \dfrac{r^{k+1}}{k+1} \right)\\
&=\; \dfrac{2r^2-r}{2(1-r)} -\dfrac{\ln (1-r)}{2}, \quad \text{and} \\ 
|z \rho''_n(z; f)| & \leq \; \dfrac{1}{2} \sum_{k=1}^{\infty} \left(k+n \right) r^{k+n-1} \,\leq \, \dfrac{1}{2} \sum_{k=1}^{\infty} (k+2) r^{k+1} = \dfrac{r^2(3-2r)}{2(1-r)^2}.
\end{align*}
\end{proof}

\section{Radius properties of the Class $\Omega$}
\setcounter{equation}{0}

The functions in $\Omega$ are convex in $\mathbb{D}_{1/2}$ \,\cite[Theorem 3.4]{peng}. However, the partial sums of $f(z) \in \Omega$ need not be convex in $\mathbb{D}_{1/2}$. For example, if $s_3(z;f)=z+\dfrac{\mu}{2}z^2+\dfrac{1-\mu^2}{4}z^3$ of $f_{\mu}$ given by \eqref{example}, we obtain
\begin{align*}
    \Re \left(1+\dfrac{z \,s_3''(z; f)}{s_3'(z; f)}\right)\ge 1-\left|\dfrac{\mu z+\dfrac{3}{2}(1-\mu^2)z^2}{1+\mu z+\dfrac{3}{4}(1-\mu^2)z^2}\right|&\ge 1- \dfrac{|\mu||z|+\dfrac{3}{2}(1-\mu^2)|z|^2}{1-|\mu||z|-\dfrac{3}{4}(1-\mu^2)|z|^2}\\
    &=\dfrac{1-2|\mu||z|-\dfrac{9}{4}(1-\mu^2)|z|^2}{1-|\mu||z|-\dfrac{3}{4}(1-\mu^2)|z|^2}.
\end{align*}
Note that $1-|\mu||z|-\dfrac{3}{4}(1-\mu^2)|z|^2>0$ in $|z|<0.964$ if $\mu=0.9$, but $\Re \left(1+\dfrac{z \,s_3''(z; f)}{s_3'(z; f)}\right)\not> 0$ for $\mu=0.9$ and $|z|\ge0.4969$. Hence, $s_3(z;f)$ is not convex in $\mathbb{D}_{1/2}$. In the next result, we obtain the radius of convexity for the partial sums of $f \in \Omega$.

\begin{theorem}\label{thm1}
Let $f(z)=z+\sum_{k=2}^{\infty}  a_k z^k \in \Omega$, and $r_{\mathcal{K}}$ denote the positive root of 
\begin{equation}\label{tr1}
(1-r)^2 \ln(1-r)+2-7r+4r^2=0.
\end{equation}
Then $s_n(z; f)$ is convex in the disk $|z|<r_{\mathcal{K}}$.
\end{theorem}
\begin{proof} 
By using Lemma \ref{lemma1}$(c)$, we obtain 
\[\Re\left(1+\dfrac{z s_2''(z; f)}{s_2'(z; f)}\right)\geq  1- \dfrac{2| a_2| |z|}{1-2| a_2| |z|} \geq  1- \dfrac{|z|}{1-|z|},\]
which is positive for $|z|<1/2$. Thus, $s_2(z; f)$ is convex in the disk $|z|<1/2:=r_1$. Further, using Lemma \ref{lemma1}$(c)$, we obtain 
\begin{align*}
\Re\left(1+\dfrac{z s_3''(z; f)}{s_3'(z; f)}\right)& =  \Re\left(\dfrac{1+4 a_2z+9 a_3z^2}{1+2 a_2z+3 a_3z^2}\right) = 2-\Re\left(\dfrac{1-3 a_3z^2}{1+2 a_2z+3 a_3z^2}\right) \\
 &\geq  2-\left|\dfrac{1-3 a_3z^2}{1+2 a_2z+3 a_3z^2}\right|  \geq 2- \dfrac{1+3| a_3||z|^2}{1-2| a_2||z|-3| a_3||z|^2}\\
&\geq  2- \dfrac{4+3|z|^2}{4-4|z|-3|z|^2} = \dfrac{4-8|z|-9|z|^2}{4-4|z|-3|z|^2}\;,
\end{align*}
which is positive for $|z|<\min\{r_2,r_{2.1}\}=r_2$, where $r_2 \approx 0.3568$ is the unique root of $9r^2+8r-4=0$ and $r_{2.1} \approx 0.6667$ is the unique root of $3r^2+4r-4=0$ in $(0,1)$. Thus, $s_3(z; f)$ is convex in the disk $|z|<r_2$.  

%Notice that the radii are decreasing as n increases. To find a fixed (stable) radius for all $n\geq 2$, we next consider the general case for $n\geq 2$.

Notice that the radii are decreasing as $n$ increases. Next, we consider the general case $n\geq 2$. To prove that $s_n(z; f)$ is convex in $|z| < r_{\mathcal{K}},$ it is sufficient to prove that for $0<|z| < r_{\mathcal{K}},$ $s'_n(z; f) \neq 0$ and
\begin{align}\label{new2.2}
\displaystyle{\Re\left(1+\dfrac{zs_n''(z; f)}{s_n'(z; f)} \right)}& = \Re\left(1+\dfrac{zf''(z)}{f'(z)}+\dfrac{\dfrac{zf''(z)}{f'(z)} \rho'_n(z; f) - z \rho''_n(z; f)}{f'(z)-\rho'_n(z; f)} \right) \\
&\geq \Re\left(1+\dfrac{zf''(z)}{f'(z)} \right)- \dfrac{\left|\dfrac{zf''(z)}{f'(z)}\right| |\rho'_n(z; f)|+|z \rho''_n(z; f)|}{|f'(z)|-|\rho'_n(z; f)|}>0, \notag
\end{align}
where $\rho_n(z; f)= \sum ^{\infty}_{k=n+1}  a_k z^k$. If $f\in \Omega$, then for $r \leq \dfrac{\sqrt{5}-1}{2},$
\begin{equation}\label{inq-convex}
\left| \dfrac{zf''(z)}{f'(z)}\right| \leq \dfrac{r}{1-r}
\end{equation}
(see \cite[Theorem 3.4]{peng}). Hence, 
\[\Re \left(1+ \dfrac{zf''(z)}{f'(z)} \right)\geq \dfrac{1-2r}{1-r}, \qquad |z| \leq \dfrac{\sqrt{5}-1}{2}. \]

To prove the inequality \eqref{new2.2}, it is sufficient to show that
\begin{equation}\label{k1}
    |\rho'_n(z; f)| < |f'(z)|, \qquad |z| < r_{\mathcal{K}}.
\end{equation}
and
\begin{equation}\label{k2}
    \dfrac{\Big|\dfrac{zf''(z)}{f'(z)}\Big| |\rho'_n(z; f)|+|z \rho''_n(z; f)|}{|f'(z)|-|\rho'_n(z; f)|} \leq \dfrac{1-2r}{1-r}, \qquad |z|=r< r_{\mathcal{K}},
\end{equation}
where $r_{\mathcal{K}}$ is the radius of convexity of $s_n(z;f)$.

Note that $\dfrac{\rho'_n(z; f)}{f'(z)}$ is analytic in $\mathbb{D}$ since $f'(z)\neq 0$ (by local mapping theorem and as $f\in \mathcal{S}$) for $0<|z|<1$. Further, by the use of the maximum modulus theorem, if \eqref{k1} holds for $|z|=r_{\mathcal{K}},$ then the same inequality holds for $0<|z|<r_{\mathcal{K}}$ when $s'_n(z; f) \neq 0$ for $0<|z|<r_{\mathcal{K}}$. Consequently, $1+\dfrac{zs''_n(z;f)}{s'_n(z;f)}$ is analytic on $|z|\leq r_{\mathcal{K}}$, and if $\Re\left(1+\dfrac{zs''_n(z;f)}{s'_n(z;f)}\right) \geq 0$ on $|z|=r_{\mathcal{K}}$, then $s_n(z; f)$ is convex in $|z|<r_{\mathcal{K}}$ by the maximum modulus theorem. But the inequality $\Re\left(1+\dfrac{zs''_n(z; f)}{s'_n(z; f)}\right) \geq 0$ on $|z|=r_{\mathcal{K}}$ follows from \eqref{k2}. If $|z|=r$, then by using Lemma \ref{lemma3a} and Lemma \ref{lemma1}$(b)$, we obtain
\begin{equation*}\label{ez}
|f'(z)|-|\rho'_n(z; f)| \geq 1-r- \left(\dfrac{2r^2-r}{2(1-r)}-\dfrac{\ln (1-r)}{2} \right) = \dfrac{2-3r}{2(1-r)}+\dfrac{\ln(1-r)}{2}, 
\end{equation*}
which is positive when $|z|<r_3,$ where $r_3\approx 0.5471.$ Hence \eqref{k1} holds when $r<r_3 \approx 0.5471.$

Now by using inequalities \eqref{ps2}, \eqref{ps3}, \eqref{inq-convex} and Lemma \ref{lemma1}$(b)$, we obtain
\begin{align*} \label{1.5}
& \dfrac{1-2r}{1-r}-\dfrac{\left|\dfrac{zf''(z)}{f'(z)}\right| |\rho'_n(z; f)|+|z \rho''_n(z; f)|}{|f'(z)|-|\rho'_n(z; f)|} \\
  & \geq  \dfrac{1-2r}{1-r}- \dfrac{\dfrac{r}{1-r}\left|\rho'_n(z; f)\right|+\left|z \rho''_n(z; f)\right|}{1-r-\left|\rho'_n(z; f)\right|} \\
 & \geq  \dfrac{1-2r}{1-r}-\dfrac{ \dfrac{r}{(1-r)} \left( \dfrac{2r^2-r}{2(1-r)}-\dfrac{\ln (1-r)}{2} \right)+\dfrac{r^2(3-2r)}{2(1-r)^2}}{1-r-\dfrac{2r^2-r}{2(1-r)}+\dfrac{\ln (1-r)}{2}}\\
&=  \dfrac{1-2r}{1-r}-\dfrac{2r^2-r(1-r)\ln(1-r)}{(1-r)(2-3r+(1-r)\ln(1-r))},
\end{align*}
which is positive for $|z|<r_4,$ where $r_4 \approx 0.3181$ is the unique root of \eqref{tr1} in $(0, (\sqrt{5}-1)/2)$. Thus $s_n(z; f), \, n \geq 2$ is convex in the disk $|z|< \min\{r_1,  r_2,  r_3, r_4 \} = r_4:=r_{\mathcal{K}}$. 
\end{proof}

 \vspace{0.2in}
The functions in $\Omega$ are starlike in $\mathbb{D}$ \,\cite[Theorem 3.1]{peng}, but the partial sums of $f(z) \in \Omega$ may not be starlike in $\mathbb{D}$. For example, if we consider $s_3(z;f)$ of $f_{\mu}$ given by \eqref{example}, then we obtain
\begin{align*}
    \Re\left(\dfrac{zs'_3(z;f)}{s_3(z;f)}\right)\ge 1- \left|\dfrac{2\mu z^2+2(1-\mu^2)z^3}{4z+2\mu z^2+(1-\mu^2)z^3} \right|&\ge 1- \dfrac{2|\mu||z|+2(1-\mu^2)|z|^2}{4-2|\mu||z|-(1-\mu^2)|z|^2}\\
    &=\dfrac{4-4|\mu||z|-3(1-\mu^2)|z|^2}{4-2|\mu||z|-(1-\mu^2)|z|^2}.
\end{align*}
Note that $4-2|\mu||z|-(1-\mu^2)|z|^2>0$ in $\mathbb D$ if $\mu=0.7$, but $\Re \left(\dfrac{zs_3'(z; f)}{s_3(z; f)}\right)\not> 0$ for $\mu=0.7$ and $|z|\ge0.9428$. Hence, $s_3(z; f)$ is not starlike in $\mathbb{D}$. In the following result, we obtain the radius of starlikeness for the partial sums of $f \in \Omega$.

\begin{theorem}\label{thm2}
Let $f(z)=z+\sum_{k=2}^{\infty}  a_k z^k \in \Omega$, and $r_{\mathcal{S}^*}$ denote the positive root of 
\begin{equation}\label{tr2}
3(1-r)(2-r)\ln(1-r)+4-4r-3r^2+2r^3=0.
\end{equation}
Then $s_n(z; f)$ is starlike in the disk $|z|<r_{\mathcal{S}^*}$.
\end{theorem}

\begin{proof} 
By using Lemma \ref{lemma1}$(c)$, we obtain 
\[\Re\left(\dfrac{z s_2'(z; f)}{s_2(z; f)}\right)=\Re\left(\dfrac{z+2  a_2 z^2}{z+ a_2 z^2}\right)\geq 1- \dfrac{| a_2||z|^2}{|z|-| a_2| |z|^2} \geq \dfrac{2(1- |z|)}{2- |z|},\]
which is positive for $|z|<1$. Thus, $s_2(z; f)$ is startlike in $|z|<1:=r_5$. Further, using Lemma \ref{lemma1}$(c)$, we obtain 
\begin{equation*}
\Re\left(\dfrac{z s_3'(z; f)}{s_3(z; f)}\right)=\Re\left(\dfrac{z+2 a_2 z^2 + 3 a_3 z^3}{z+ a_2 z^2+ a_3 z^3}\right) \geq  \dfrac{4-4|z|-3|z|^2}{4-2|z|-|z|^2},
\end{equation*}
which is positive for $|z|<r_6,$ where $r_6 \approx 0.6667$ is the unique root of $3r^2+4r-4=0$ in $(0, 1)$. Thus, $s_3(z; f)$ is starlike in $|z|<r_6$. 

Notice that the radii are decreasing as $n$ increases. Next, we consider the general case $n\geq 2$. To prove that $s_n(z; f)$ is starlike in $|z| < r_{\mathcal{S}^*},$ it is sufficient to prove that for $0<|z| < r_{\mathcal{S}^*},$ $s_n(z; f) \neq 0$ and
\begin{align}\label{new2.7}
\Re\left(\dfrac{zs'_n(z;f)}{s_n(z;f)}\right) &= \Re\left(\dfrac{zf'(z)}{f(z)}+\dfrac{\displaystyle{\dfrac{zf'(z)}{f(z)} \rho_n(z; f)-z \rho'_n(z; f)}}{f(z)-\rho_n(z; f)}\right) \\
&\geq \Re\left(\dfrac{zf'(z)}{f(z)} \right)- \dfrac{\left|\dfrac{zf'(z)}{f(z)}\right| |\rho_n(z; f)|+|z \rho'_n(z; f)|}{|f(z)|-|\rho_n(z; f)|}>0. \notag
\end{align}
If $f\in\Omega$, then Lemma \ref{lemma3} shows that
\begin{equation}\label{newb}
\Re \left( \dfrac{zf'(z)}{f(z)}\right) \geq \dfrac{2(1-r)}{2-r}, \qquad |z|=r<1.
\end{equation}

To prove the inequality \eqref{new2.7}, it is sufficient to show that
\begin{equation}\label{c}
    |\rho_n(z; f)| < |f(z)|, \quad |z| < r_{\mathcal{S}^*},
\end{equation}
and
\begin{equation}\label{d}
    \dfrac{\displaystyle{\left|\dfrac{zf'(z)}{f(z)}\right| |\rho_n(z; f)|+|z \rho'_n(z; f)|}}{|f(z)|-|\rho_n(z; f)|} \leq \dfrac{2(1-r)}{2-r}, \qquad |z|=r< r_{\mathcal{S}^*},
\end{equation}
where $r_{\mathcal{S}^*}$ is the radius of starlikeness of $s_n(z;f)$.

Note that $\dfrac{\rho_n(z; f)}{f(z)}$ is analytic in $\mathbb{D}$ since $f(z)\neq 0$ for $0<|z|<1$. Further, by the use of the maximum modulus theorem, if \eqref{c} holds for $|z|=r_{\mathcal{S}^*},$ then the same inequality holds for $0<|z|<r_{\mathcal{S}^*}$ when $s_n(z;f) \neq 0$ for $0<|z|<r_{\mathcal{S}^*}$. Consequently, $\dfrac{zs'_n(z;f)}{s_n(z;f)}$ is analytic on $|z|\leq r_{\mathcal{S}^*}$ and if $\Re\left(\dfrac{zs'_n(z;f)}{s_n(z;f)}\right) \geq 0$ on $|z|=r_{\mathcal{S}^*}$, then $s_n(z;f)$ is starlike in $|z|<r_{\mathcal{S}^*}$ by the maximum modulus theorem. But the inequality $\Re\left(\dfrac{zs'_n(z;f)}{s_n(z;f)}\right) \geq 0$ on $|z|=r_{\mathcal{S}^*}$ follows by using \eqref{newb} and \eqref{d} in \eqref{new2.7}. Using \eqref{ps1} and Lemma \ref{lemma1}$(a)$, we obtain
\begin{equation*}\label{e}
|f(z)|-|\rho_n(z; f)| \geq  r-\dfrac{r^2}{2}+\dfrac{r}{2}\left(\ln (1-r)+r \right), 
\end{equation*}
which is positive when $|z|<r_7,$ where $r_7 \approx 0.8647.$ Hence \eqref{c} holds when $|z|<r_7.$

By using the inequalities \eqref{ps1}, \eqref{ps2}, \eqref{newb}, \eqref{c}, and \eqref{d}, we obtained 
\begin{align*}
\Re\left(\dfrac{zs'_n(z;f)}{s_n(z;f)}\right) &> \dfrac{2(1-r)}{2-r}-\dfrac{\left|\dfrac{zf'(z)}{f(z)}\right| |\rho_n(z; f)|+|z \rho'_n(z; f)|}{|f(z)|-|\rho_n(z; f)|}\\
&\geq  \dfrac{2(1-r)}{2-r}- \dfrac{\dfrac{2}{2-r}\left(-\dfrac{r\ln(1-r)}{2}-\dfrac{r^2}{2}\right)+\dfrac{2r^3-r^2}{2(1-r)}-\dfrac{r\ln(1-r)}{2}}{\displaystyle{r-\dfrac{r^2}{2}+\dfrac{r^2}{2}+\dfrac{r\ln(1-r)}{2}}}\notag\\
&= \dfrac{2(1-r)}{2-r}+ \dfrac{\dfrac{1}{2-r}\left( r\ln(1-r)+ r^2 \right)-\dfrac{2r^3-r^2}{2(1-r)}+\dfrac{r\ln(1-r)}{2}}{r+\dfrac{r\ln(1-r)}{2}},\notag
\end{align*}
which is positive for $|z|<r_8,$ where $r_8 \approx 0.4899$ is the unique root of \eqref{tr2}. Thus $s_n(z; f), \, n \geq 2$ is starlike in the disk $|z|< \min\{r_5,  r_6,  r_7, r_8\}=r_8:=r_{\mathcal{S}^*}$.  
\end{proof}

\vspace{0.2in}
The functions in $\Omega$ are close-to-convex in $\mathbb{D}$ \,\cite[Lemma 3.2]{{hesam}}, but the partial sums of $f(z) \in \Omega$ may not be close-to-convex in $\mathbb{D}$. For example, if we consider $s_3(z;f)$ of $f_{\mu}$ given by \eqref{example}, we obtain
\[\Re\left(s'_3(z;f)\right)\ge 1- \left|\mu z+\dfrac{3}{4}(1-\mu^2)z^2\right|\ge 1-|\mu||z|-\dfrac{3}{4}(1-\mu^2)|z|^2,\]
is positive when $\mu=0.7$ and $|z|<0.9428$. Thus, $\Re\left(s'_3(z;f)\right)\not>0$ for $\mu=0.7$ and $|z|\ge 0.9428$. Hence, $s_3(z; f)$ is not close-to-convex in $\mathbb{D}$.
In the next result, we obtain the radius of close-to-convexity for the partial sums of $f \in \Omega$.

\begin{theorem}\label{thm3}
Let $f(z)=z+\sum_{k=2}^{\infty}  a_k z^k \in \Omega$, and $r_{\mathcal{S}}$ denote the positive root of 
\begin{equation}\label{tr3}
(1-r)\ln(1-r)+2-3r=0.
\end{equation}
Then $s_n(z; f)$ is close-to-convex in the disk $|z|<r_{\mathcal{S}}$.
\end{theorem}

\begin{proof} 
Using Lemma \ref{lemma1}$(c)$, we have 
\[\Re\left( s_2'(z; f)\right)\geq 1-2| a_2||z| \geq 1-|z|,\]
which is positive for $|z|<1$. Thus, $s_2(z,f)$ is close-to-convex in the disk $|z|<1:=r_9$.
Further, using Lemma \ref{lemma1}$(c)$, we have 
\[\Re\left(s_3'(z; f)\right)\geq 1-2| a_2||z|-3| a_3||z|^2 \geq 1-|z|-\dfrac{3}{4}|z|^2, \]
which is positive for $|z|<r_{10},$ where $r_{10} \approx 0.66667$ is the unique root of $3r^2+4r-4=0$ in $(0, 1)$. Thus, $s_3(z; f)$ is close-to-convex in the disk $|z|<r_{10}$. 

Next, we consider the general case $n\geq 2$. Following the proof of \cite[Lemma 3.2]{hesam}, we have
\begin{equation}\label{e3}
\Re(f'(z))\geq 1-r.
\end{equation}
Using \eqref{ps2} and \eqref{e3}, we obtain
\begin{align*}
\Re(s'_n(z; f)) &\geq \Re(f'(z))-\Re(\rho'_n(z; f))\\
& \geq 1-r - \left(\dfrac{2r^2-r-(1-r)\ln(1-r)}{2(1-r)}\right)\\
&= \dfrac{2(1-r)^2 -2r^2+r+(1-r)\ln(1-r)}{2(1-r)}\\
&= \dfrac{2-3r+(1-r)\ln(1-r)}{2(1-r)},
\end{align*}
which is positive for $|z|<r_{11},$ where $r_{11} \approx 0.5471$ is the unique root of \eqref{tr3}. Thus, $s_n(z; f), \, n \geq 2$ is close-to-convex in the disk $|z|<\min\{r_9, r_{10}, r_{11}\}=r_{11}:=r_{\mathcal{S}}$.  
\end{proof}
\vspace{0.2in}

\section{Inequalities on the partial sums of the Class $\Omega$}
\setcounter{equation}{0}

If $f(z)\in \Omega$, then following the proof of Theorem \ref{thm3}, we observe that the radius of close-to-convexity is decreasing as $n$ is increasing, but it remains constant for all $n\geq 12$, which can be seen in Figure \ref{fig:ctcg} of $n$ and approximate value of radius of disk in which $\Re(s'_n(z; f))>0, \; f(z) \in \Omega$. This section is devoted to establishing this fact.

\begin{figure}[htbp]
\centering
\includegraphics[scale=0.8]{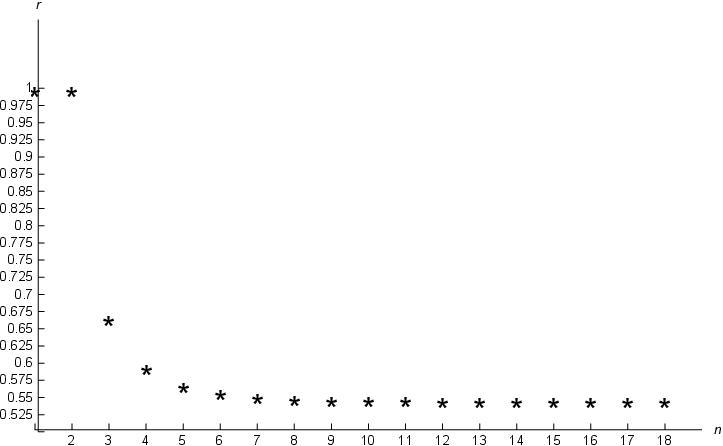}
  \caption{Graph between $n$ and approximate value of radius of disk in which $\Re(s'_n(z; f))>0$.}
  \label{fig:ctcg}
\end{figure}

\begin{theorem}\label{thm11}
Let $f(z) \in \Omega$. Then, for each $n\geq 2$, we have
\begin{equation}\label{apO1}
\bigg|\dfrac{s_n'(z,f)}{f'(z)}-1 \bigg| \leq |z|^n\left( \dfrac{n+1}{2n} + A_n \dfrac{|z|}{1-|z|}\right),\quad |z|=r<1,
\end{equation}
where $A_n=\dfrac{\sqrt{2r-r^2}}{2(1-r)}(n+1+\ln(n-1)+\gamma)$ and $\gamma \approx 0.57722$ is the Euler-Mascheroni constant.
\end{theorem}
\begin{proof}
Let $f(z) \in \Omega$. Since $f(z)\in \Omega$, it is univalent in $\mathbb{D}$, and so $f'(z)$ is non-vanishing in $\mathbb{D}$. Hence, $1/f'(z)$ can be represented in the form 
\[\dfrac{1}{f'(z)}=1+d_1z+d_2z^2+\dots \;,\]
where $d_1=-2 a_2$, $d_2=4 a_{2}^2-3 a_3$ and so on, for some complex coefficients $d_n$, $n\geq 1$.\\
Now considering the identity $f'(z)\displaystyle{\dfrac{1}{f'(z)}}\equiv 1 $, we have
\[(1+2 a_2z+3 a_3z^2+\dots)(1+d_1z+d_2z^2+\dots)= 1, \quad z\in \mathbb{D}. \]
From this, we get
\begin{equation}\label{ee1}
m a_m + \sum_{k=1}^{m-1}(m-k) a_{m-k}d_k=0,
\end{equation}
where $m\geq 2$ and $ a_1=1$.\\
Now using \eqref{ee1} and the series representation of the partial sum $s_n(z,f)$, we get
\begin{align}\label{ee2}
\dfrac{s_n'(z,f)}{f'(z)}&=(1+2 a_2 z+\dots +n a_n z^{n-1})(1+d_1z+d_2z^2+\dots) \\
&\equiv  1+c_nz^n+c_{n+1}z^{n+1}+\dots \;, \notag
\end{align}
where 
\begin{align}\label{ee3}
c_n&=n a_nd_1+(n-1) a_{n-1}d_2+\dots +2 a_2d_{n-1}+ d_n,\\
c_{n+1}&=n a_nd_2+(n-1) a_{n-1}d_3+\dots +2 a_2d_n +d_{n+1},\notag \\
c_{n+2}&=n a_nd_3+(n-1) a_{n-1}d_4+\dots +2 a_2d_{n+1}+d_{n+2}, \notag
\end{align}
and so on.\\
Putting $m=n+1$ in \eqref{ee1} and comparing with \eqref{ee3}, we get,
\begin{equation*}
c_n=-(n+1) a_{n+1}.
\end{equation*}
Now using Lemma \ref{lemma1}$(c)$, we get
\begin{equation}\label{ee4}
|c_n|\leq \dfrac{n+1}{2n}, \qquad n\ge 1.
\end{equation}

Let $M$ be a constant such that
\begin{equation}\label{ee5}
\left|\dfrac{1}{f'(z)}\right|\le M, \qquad z\in \mathbb D.
\end{equation}
Then, using Lemma \ref{lemma1}$(b)$, we obtain that $M\ge1$. Now using \eqref{ee5}, for $z=re^{i\theta}$, we get
\begin{align}\label{ee6}
\dfrac{1}{2\pi}\int_0^{2\pi}\bigg|\dfrac{1}{f'(re^{i \theta})}\bigg|^2 d\theta&=\dfrac{1}{2\pi}\int_0^{2\pi}|1+d_1 re^{i \theta}+d_2 r^2 e^{2i \theta}+\dots|^2 d\theta \notag \\
&=1+ \sum_{k=1}^{\infty}|d_k|^2 r^{2k} \leq M^2.
\end{align}
Allowing $r\rightarrow 1^-$, we obtain
\[ \sum_{k=1}^{\infty}|d_k|^2 \leq M^2 -1.\]
From this, we conclude that
\begin{equation}\label{ee7}
|d_k| \leq \sqrt{M^2 -1},
\end{equation}
for each $k \in \mathbb{N}$.

Observing the series $c_{n+1}$, $c_{n+2}$, and so on, for $m\geq n+1$, we obtain
\begin{equation*}
c_m=n a_n d_{m-n+1}+(n-1) a_{n-1} d_{m-n+2}+\dots +2 a_2 d_{m-1}+d_m.
\end{equation*}
Using Lemma \ref{lemma1}$(c)$ and \eqref{ee7}, we get for $m\geq n+1$,
\begin{align}\label{ee71}
|c_m|&\leq  \sum_{k=2}^{n} \dfrac{k}{2(k-1)}| d_{m-k+1}|+| d_m| \\
&\leq \left(1+\dfrac{1}{2}\sum_{k=2}^{n} \dfrac{k}{k-1}\right)\sqrt{M^2-1} \notag \\
&= \left(1+\dfrac{1}{2}\sum_{k=2}^{n}\left( 1+ \dfrac{1}{k-1}\right)\right)\sqrt{M^2-1} \notag \\
&= \left(1+\dfrac{n-1}{2}+\dfrac{1}{2}\sum_{k=1}^{n-1} \dfrac{1}{k}\right)\sqrt{M^2-1} \notag \\
&\approx (n+1+\ln(n-1)+\gamma)\;\dfrac{\sqrt{M^2-1}}{2},\notag
\end{align}
where $\gamma \approx 0.57722$ is the Euler-Mascheroni constant. Further using \eqref{ee5}, for $m\geq n+1$, we get
\begin{equation}\label{ee8}
|c_m|\leq \dfrac{\sqrt{2r-r^2}}{2(1-r)}(n+1+\ln(n-1)+\gamma):=A_n.
\end{equation}
Now using \eqref{ee4} and \eqref{ee8} in \eqref{ee2}, we get
\begin{align}
\bigg| \dfrac{s_n'(z,f)}{f'(z)}-1 \bigg| &= |c_nz^n+c_{n+1}z^{n+1}+\dots|  \\
&\leq  |c_n||z^n|+|c_{n+1}||z|^{n+1}+\dots \notag \\
&\leq  |z|^n\left(\dfrac{n+1}{2n}+A_n|z|+A_n|z|^2+\dots \right) \notag \\
&= |z|^n \left( \dfrac{n+1}{2n}+A_n\dfrac{|z|}{1-|z|}\right),\notag
\end{align}
for $|z|=r<1$. This completes the proof of the theorem.
\end{proof}

\begin{theorem}\label{thm12}
Let $f(z)\in \Omega$. Then, for each $r\in (0,1)$ and $n\geq 2$, we have 
\begin{equation}\label{apO2}
\bigg|\dfrac{s_n'(z,f)}{f'(z)}-1 \bigg| \leq |z|^n\left( \dfrac{n+1}{2n} + B_n \dfrac{|z|}{r-|z|}\right),\quad |z|<r,
\end{equation}
where $B_n=\dfrac{\sqrt{2r-r^2}}{2(1-r)r^n}(n+1+\ln(n-1)+\gamma)$.
\end{theorem}

\begin{proof}
Let $f(z)\in \Omega$. Following the proof of previous Theorem \ref{thm11} and using \eqref{ee71}, we have
\begin{align}\label{ee9}
|c_m| &\leq  \sum_{k=2}^{n} \dfrac{k}{2(k-1)}| d_{m-k+1}|+| d_m|,  \notag \\
&\leq  \sum_{k=2}^{n}\left( \dfrac{k}{2(k-1)r^{m-k+1}}| d_{m-k+1}|\right)+\dfrac{1}{r^m}| d_m|,
\end{align}
for any arbitrary fixed $r\in (0,1)$ and $m\geq n+1$. From \eqref{ee6}, we get
\begin{equation}\label{ee10}
|d_k|r^k \leq \sqrt{M^2-1},
\end{equation}
for each $k \in \mathbb{N}$.
Now using the inequality \eqref{ee10} in \eqref{ee9}, we get
\begin{align}
|c_m| &\leq  \left(1+\dfrac{1}{2}\sum_{k=2}^{n} \dfrac{k}{k-1}\right)\dfrac{\sqrt{M^2-1}}{r^m} \notag \\
&= \left(1+\dfrac{n-1}{2}+\dfrac{1}{2}\sum_{k=1}^{n-1} \dfrac{1}{k}\right)\dfrac{\sqrt{M^2-1}}{r^m} \notag \\
&\approx (n+1+\ln(n-1)+\gamma)\dfrac{\sqrt{M^2-1}}{2r^m}, \notag
\end{align}
for $r \in (0,1)$ and $\gamma \approx 0.57722$. Now using \eqref{ee5} here, for $m\geq n+1$ we have
\begin{equation}\label{ee11}
|c_m|\leq \dfrac{1}{r^m}\left(\dfrac{\sqrt{2r-r^2}}{2(1-r)}(n+1+\ln(n-1)+\gamma)\right):=\dfrac{1}{r^m} D_n.
\end{equation}
Now using \eqref{ee4} and \eqref{ee11} in \eqref{ee2}, we get
\begin{align}
\bigg| \dfrac{s_n'(z,f)}{f'(z)}-1 \bigg|
&\leq |c_n||z^n|+|c_{n+1}||z|^{n+1}+\dots \notag \\
&\leq |z|^n\left(\dfrac{n+1}{2n}+\dfrac{D_n}{r^n}\left(\dfrac{|z|}{r}+\dfrac{|z|^2}{r^2}\dots \right)\right) \notag \\
&= |z|^n \left( \dfrac{n+1}{2n}+\dfrac{D_n}{r^n}\dfrac{|z|/r}{1-|z|/r}\right)\notag \\
&= |z|^n \left( \dfrac{n+1}{2n}+B_n\dfrac{|z|}{r-|z|}\right), \quad |z|<r, \notag
\end{align}
where $B_n=\dfrac{\sqrt{2r-r^2}}{2(1-r)r^n}(n+1+\ln(n-1)+\gamma)$ and $\gamma \approx 0.57722$ is the Euler-Mascheroni constant. This completes the proof of the theorem.
\end{proof}

\begin{theorem}\label{thm13}
Let $f(z) \in \Omega$. Then $\Re(s'_n(z,f))>0$ in $|z|\leq r_0$ for $n\geq 12$, where $r_0=0.547$.
\end{theorem}
\begin{proof}
If $f(z)\in \Omega$, then in view of \cite[Equation 3.6]{peng},
\[f'(z)=1+\dfrac{1}{2}z\phi(z)+\dfrac{1}{2}z\int_0^1\phi(zt)dt,\]
where $|\phi(z)| \le 1$. Hence
\[|f'(z)-1|=\left|\dfrac{1}{2}\phi(z)+\dfrac{1}{2}z\int_0^1\phi(zt)dt\right|<|z|.\]
Therefore, using the maximum modulus principle for $|z|\leq r_0=0.547$, we obtain
\begin{equation}\label{ee12}
\underset{|z|= r_0}{\max} |\arg f'(z)|\leq \sin^{-1}(r_0)< 33.16^{\circ}.
\end{equation}
Now taking $r=0.99$ in inequality (\ref{apO2}), we get
\begin{equation*}
\bigg|\dfrac{s_n'(z;f)}{f'(z)}-1 \bigg| \leq C_n,\quad |z|< r_0.
\end{equation*}
where, 
$$C_n=(0.547)^n\left( \dfrac{n+1}{2n} + \dfrac{61.735}{(0.99)^n}(n+1+\ln(n-1)+\gamma)\right).$$
Now, using the maximum modulus principle, we obtain
\begin{equation}\label{ee13}
\underset{|z|= r_0}{\max} \bigg|\arg \dfrac{s_n'(z;f)}{f'(z)}\bigg|\leq \sin^{-1}(C_n), \quad |z|< r_0.
\end{equation}
Using \eqref{ee12} and \eqref{ee13}
$$|\arg s'_n(z;f)|\leq |\arg f'(z)| +  \bigg|\arg \dfrac{s_n'(z;f)}{f'(z)}\bigg|< 33.16^{\circ} + \sin^{-1}(C_n)<\dfrac{\pi}{2}.$$
This inequality holds true if $\sin^{-1}(C_n)\leq 56.84^{\circ}$. Now the result follows, as the last inequality holds true for $n\geq 12$.
\end{proof}
\vspace{0.2in}

\section{Properties of the Class $\Omega_{\lambda}$}
\setcounter{equation}{0}
\indent
The functions in class $\Omega_{\lambda}$ are starlike in $\mathbb{D}$ when $\lambda\in \left(0,1/2\right]$. If $f(z)$ has the form (\ref{egOA}), then
$$ \Re \left( \dfrac{zf'(z)}{f(z)} \right)\ge 1-\left| \dfrac{2\lambda z+ 2\lambda z^2}{4+2\lambda z +\lambda z^2} \right| \ge 1-\; \dfrac{2\lambda |z|+ 2\lambda |z|^2}{4-2\lambda |z| -\lambda |z|^2}> 1- \;\dfrac{4\lambda}{4-3\lambda}\ge 0, \quad(z\in \mathbb D), $$ whenever $\lambda\le 4/7$. That is, $f$ is not starlike when $\lambda>4/7$. Thus, all functions in $\Omega_{\lambda}$ are not starlike when $\lambda>1/2$. Therefore, in the next result, we obtain a subdisk of $\mathbb{D}$ where all the functions in class $\Omega_{\lambda}$ are starlike when $\lambda>1/2$. 

\begin{theorem}\label{thm4}
The radius of starlikeness of $\Omega_{\lambda}$ is $\dfrac{1}{2\lambda}$. Further, if $f(z)\in \Omega_{\lambda}$, then
\begin{align}
|z|-\lambda |z|^2 \leq &\; |f(z)| \leq \;|z|+\lambda |z|^2,\label{g1.01} \\ 
1-2\lambda |z| \leq &\;  |f'(z)| \leq \;1+2\lambda |z|.\label{g1.02}
\end{align}
Equality occurs in both inequalities if and only if $f(z)=z+\lambda \eta z^2$ for $z\in \mathbb{D}$ and $z\neq 0$, where $|\eta|=1$.
\end{theorem}

\begin{proof}
If $f(z)\in \Omega_{\lambda}$, then from \eqref{oa1} we obtain
\begin{equation} \label{g1.1}
zf'(z)-f(z)= \lambda z^2 \phi(z),
\end{equation}
where $\phi(z)$ is analytic in $\mathbb{D}$ such that $|\phi(z)|\leq 1$. Equation \eqref{g1.1} is equivalent to 
\begin{equation}\label{g3.1}
  f(z) = z+\lambda z \int_{0}^{z} \phi(\zeta)d\zeta = z+\lambda z^2 \int_{0}^{1} \phi(zt)dt.
\end{equation}
Using the bound of $\phi$ for $z\in \mathbb{D}$ and setting the inequality to both sides of $|f(z)|$, we get the estimate of \eqref{g1.01}. Further,
\begin{equation}\label{g3.2}
f'(z) = 1 + \lambda z \phi(z)+\lambda z\int_{0}^{1} \phi(zt)dt.
\end{equation}
Now using the bound of $\phi$ for $z\in \mathbb{D}$ and setting the inequality to both sides of $|f'(z)|$, we get the estimate of \eqref{g1.02}. The equality in both estimates occurs if and only if $|\int_{0}^{1} \phi(zt)dt|=1$. It follows that $|\phi(zt)|\equiv 1$ for $0\leq t \leq 1$. Thus $\phi(z)\equiv \eta \;(|\eta|=1)$ by the maximum modulus theorem, and so $f(z)=z+\lambda \eta z^2$ where $|\eta|=1$.
 
Moreover, using the lower bound of \eqref{g1.01}, we obtain
\begin{equation*}\label{g4}
\bigg|\dfrac{zf'(z)}{f(z)} -1 \bigg| \leq \dfrac{\lambda |z|^2}{|z|-\lambda |z|^2}<1  
\end{equation*}  
if and only if $|z|<\dfrac{1}{2\lambda}$. Thus, $f$ is starlike in the disk $|z|<\dfrac{1}{2\lambda}$.
\end{proof}

 \vspace{0.2in}
The functions in class $\Omega_{\lambda}$ are univalent in $\mathbb{D}$ and convex in $\mathbb{D}_{1/2}$ when $\lambda\in \left(0,1/2\right]$. If $f(z)$ has the form (\ref{egOA}), then $$\Re(f'(z))\ge 1-\left|\lambda z+ \dfrac{3\lambda}{4}z^2\right|>1- \lambda- \dfrac{3\lambda}{4}\ge 0,\quad(z\in \mathbb D),$$
whenever $\lambda\le 4/7$. That is, $f$ is not univalent when $\lambda>4/7$. Thus, all functions in $\Omega_{\lambda}$ are not univalent when $\lambda>1/2$. Similarly, if $f(z)$ has the form (\ref{egOA}), then 
\begin{align*}
    \Re\left(1+ \dfrac{zf''(z)}{f'(z)} \right)&\ge 1-\left| \dfrac{\lambda z + \dfrac{3\lambda}{2} z^2}{1+\lambda z + \dfrac{3 \lambda }{4}z^2} \right|\ge 1- \dfrac{\lambda |z| + \dfrac{3\lambda}{2} |z|^2}{1-\lambda |z| - \dfrac{3 \lambda}{4}|z|^2}\\
    &> 1-\dfrac{\lambda + \dfrac{3\lambda}{2}}{1-\lambda - \dfrac{3 \lambda}{4}}\;\ge \;0,\qquad(z\in \mathbb D), 
\end{align*}
 whenever $\lambda\le 4/17$. That is, $f$ is not convex whenever $\lambda>4/17$. Thus, all functions in $\Omega_{\lambda}$ are not convex when $\lambda>1/2$. Therefore, we obtain subdisks of $\mathbb{D}$ and $\mathbb{D}_{1/2}$ where the functions in class $\Omega_{\lambda}$ are univalent and convex, respectively, when $\lambda>1/2$.

\begin{theorem}\label{thm5}
Let $f(z)\in \Omega_{\lambda}$. Then $f$ is univalent in $|z|<\dfrac{1}{2\lambda}$ and convex in $|z|<\dfrac{1}{4\lambda}$.
\end{theorem}
\begin{proof}
Let $f\in \Omega_{\lambda}$. Using (\ref{g3.2}), we obtain
\begin{equation}
|f'(z)-1| \leq 2\lambda |z| <1.
\end{equation}
if $|z|< \dfrac{1}{2\lambda}$. Thus, $f$ is univalent in $|z|<\dfrac{1}{2\lambda}$. Again, using \eqref{g3.2}, we obtain
\begin{equation}\label{g5}
f''(z)=\lambda \phi(z) + \lambda (z \phi(z))'.
\end{equation}
Now using \cite[Dieudonn\'{e}'s Lemma (page 198)]{duren} for $z\phi(z)$, we obtain 
\begin{equation} \label{g6}
|(z\phi(z))'| \le |\phi(z)|+ \dfrac{|z|^2-|z\phi(z)|^2}{|z|(1-|z|^2)}. 
\end{equation}
Using \eqref{g6} in \eqref{g5} and the bound of $\phi(z)$, we obtain
\begin{equation}\label{g7}
|f''(z)|\leq \lambda |\phi(z)| + \lambda \left( |\phi(z)|+ \dfrac{|z|^2-|z\phi(z)|^2}{|z|(1-|z|^2)} \right)\le 2\lambda.
\end{equation}
Now using the lower bound of \eqref{g1.02} and using \eqref{g7}, we obtain
\[ \left| \dfrac{zf''(z)}{f'(z)} \right|\leq \dfrac{2\lambda|z|}{1-2\lambda|z|}<1, \]
if $|z|<\dfrac{1}{4\lambda}$. Thus, $f$ is convex in $|z|<\dfrac{1}{4\lambda}$.
\end{proof}

\begin{theorem}\label{thm6}
Let $f(z) \in \mathcal{A}$. If $\,|f''(z)|\leq 2\lambda\,$ then $\,f\in \Omega_{\lambda}$. The number $2\lambda$ is the best possible.
\end{theorem}

\begin{proof}
Let $|f''(z)|\leq 2\lambda$. If $g(z)=zf'(z)-f(z)$, then $g'(z)=zf''(z)=2\lambda z \phi(z)$, where $\, |\phi(z)| \leq 1$. Note that
\[g(z)=2\lambda\int_0^z \zeta \phi(\zeta)d\zeta= 2\lambda z^2\int_0^1 t \phi(zt)dt.\]
Therefore,
\[|g(z)|=\left|2\lambda z^2\int_0^1 t \phi(zt)dt\right| <2\lambda\int_0^1 t dt=\lambda,\qquad z\in \mathbb{D}. \]
This implies $f(z)\in \Omega_{\lambda}$. 

Further, let $|f''(z)|\leq \eta$ and $\eta>2\lambda$. Considering $f(z)=z+\dfrac{1}{2}\eta z^2$, we get $|f''(z)|\leq \eta$. But $f(Z)$ is not univalent in $\mathbb{D}_{1/2\lambda}$ as $f'(z)$ vanishes at $\;-\,\dfrac{1}{\eta} \in \mathbb{D}_{1/2\lambda}$. Therefore, the number $2\lambda$ is best possible.
\end{proof}

\begin{theorem}\label{thm7}
Let $f(z)\in \mathcal{A}$. If 
\begin{equation}\label{g12}
|z^2f''(z)+zf'(z)-f(z)|\leq 3\lambda,
\end{equation}
then $f(z)\in \Omega_{\lambda}$. The number $3\lambda$ is best possible.
\end{theorem}
\begin{proof}
Let $f(z)\in \mathcal{A}$ satisfy the inequality \eqref{g12}. Then
\[(z^2f'(z)-zf(z))'=3\lambda z^2 \phi(z),\]
where $\phi$ is analytic in $\mathbb{D}$ such that $|\phi(z)|\leq 1$. Thus,
\[ z^2 f'(z)-zf(z)=3\lambda\int_0^z \zeta^2 \phi(\zeta) d\zeta=3\lambda z^3\int_0^1 t^2 \phi(zt) dt,\]
consequently,
\[ |zf'(z)-f(z)|=\left| 3\lambda z^2\int_0^1 t^2 \phi(zt) dt\right|< 3\lambda\int_0^1 t^2 dt=\lambda, \quad z\in \mathbb{D}. \]
This implies $f(z) \in \Omega_{\lambda}$. Let $|z^2 f''(z)+zf'(z)-f(z)|\leq \eta $ and $\eta>3\lambda$. Considering $f(z)=z+\dfrac{1}{3}\eta z^2$, we obtain $|z^2f''(z)+zf'(z)-f(z)|= |\eta z^2|\leq \eta, \, z\in \mathbb{D}$. But $f(z)$ is not univalent in $\mathbb{D}_{1/2\lambda}$ as $f'(z)$ vanishes at $\;-\,\dfrac{3}{2 \eta} \in \mathbb{D}_{1/2\lambda}$. Thus, the number $3\lambda$ is best possible.
\end{proof}

\begin{theorem}\label{thm9}
Let $f_1(z)$, $f_2(z)\in \Omega_{\lambda}$. Then $(f_1 \ast f_2)(z)\in\Omega_{\lambda}$, if $\lambda\leq 1$.
\end{theorem}

\begin{proof}
Let $f_1(z)$, $f_2(z)\in \Omega_{\lambda}$. Then
\begin{equation*}
f_1(z)=z+\lambda \int_{0}^{1}z^2 \phi_1(zt)dt, \qquad f_2(z)=z+\lambda \int_{0}^{1}z^2 \phi_2(zt)dt, 
\end{equation*}
where $|\phi_1(z)|\leq 1$ and $|\phi_2(z)|\leq 1$. Now
\begin{align*}
f_1(z)\ast f_2(z)&= \left(z+\lambda \int_{0}^{1}z^2 \phi_1(zt)dt \right)\ast f_2(z)  \\
&= z+ \lambda \int_{0}^{1}\bigg(z^2 \phi_1(zt) \bigg)\ast \left(z+\lambda \int_{0}^{1}z^2 \phi_2(z\zeta)d\zeta \right) dt \notag \\
&= z+\lambda^2 \int_{0}^{1} \bigg(\int_{0}^{1} z^2 [\phi_1(zt)]\ast [\phi_2(z\zeta)] d\zeta \bigg) dt \notag \\
&= z+\lambda \int_{0}^{1} z^2 \phi_0(zt)dt, \notag
\end{align*}
where $\phi_0(z)=\lambda \displaystyle{ \int_{0}^{1}[\phi_1(z)]\ast [\phi_2(z\zeta)] d\zeta}$.
Now, using {\cite[Lemma 3.6]{peng}}, we obtain
\[ |\phi_0(z)|\leq \lambda  \int_{0}^{1}| \phi_1(z)\ast \phi_2(z\zeta)| d\zeta \leq \lambda. \]
The result follows if $\lambda\leq 1$.
\end{proof}

\begin{theorem}\label{thm10}
Let $f(z)=z+\sum_{k=2}^{\infty} a_k z^k \in \mathcal{A}$. If
\begin{equation}\label{g10}
\sum_{k=2}^{\infty}(k-1) |a_k|< \lambda,
\end{equation}
then $f(z)\in \Omega_{\lambda}$. In particular, for $f(z)\in \Omega_{\lambda}$, $|a_k|\leq \dfrac{\lambda}{k-1}$ for all $k \geq 2$.
\end{theorem}

\begin{proof}
Using (\ref{g10}) we obtain
\begin{equation*}
|zf'(z)-f(z)|=\bigg| \sum_{k=2}^{\infty}ka_k z^k-\sum_{k=2}^{\infty}a_k z^k\bigg| \leq \sum_{k=2}^{\infty}(k-1)|a_k| < \lambda.
\end{equation*}
Hence $f(z)\in \Omega_{\lambda}$. Now, for $f(z)\in \Omega_{\lambda}$
\begin{equation*} \label{g8}
zf'(z)-f(z)= \sum_{k=2}^{\infty}(k-1)a_k z^k \prec \lambda z^2 \prec \lambda z , \quad z\in \mathbb{D}.
\end{equation*}
Since $\lambda z$ is convex univalent in $\mathbb{D}$, we obtain the particular result using Lemma \ref{l2}. 
\end{proof}

%\section*{Acknowledgement}
%The authors are very grateful to the referees for appropriate and constructive suggestions to improve the manuscript. 


\begin{thebibliography}{99}

\bibitem{wani} A. Swaminathan and L. A. Wani, {\it Sufficient conditions and radii problems for a starlike class involving a differential inequality}, Bull. Korean Math. Soc. {\bf 57(6)}(2020), 1409--1426.     

\bibitem{bshouty} D. Bshouty and W. Hengartner, {\it Criteria for extremality of the koebe mapping}, Proc. Amer. Math. Soc. {\bf 111(2)}(1991), 403--411.

\bibitem{szego} G. Szeg\"{o}, {\it Zur Theorie der schlichten Abbildungen},  Math. Ann. {\bf 100(1)}(1928), 188--211.

\bibitem{hesam} H. Mahzoom and R. Kargar, {\it Further results for a subclass of univalent functions related with differential inequality}, Indian J. Pure Appl. Math. {\bf 52}(2021), 205--215.

\bibitem{obradovic} M. Obradovi\'{c} and  Z. Peng, {\it Some new results for certain classes of univalent functions}, Bull. Malays. Math. Sci. Soc. {\bf 41} (2017), 1623--1628.

\bibitem{robertson} M. S. Robertson, {\it The partial sums of multivalently star-like functions}, Ann. of Math. (2) {\bf 42(4)}(1941), 829--838.

\bibitem{duren} P. L. Duren, {\it Univalent functions}, Grundlehen der Mathematischen Wissenschaften {\bf 259} Springer- Verlag, New York, 1983.

\bibitem{rusch2} St. Ruscheweyh, {\it Extension of Szeg\"{o}'s theorem on the sections of univalent functions}, SIAM J. Math. Anal. {\bf 19(6)}(1988), 1442--1449.

\bibitem{rusch} St. Ruscheweyh and T. Sheil-Small, {\it Hadamard products of schlicht function and the P\'{o}lya-Schoenberg conjecture}, Comment. Math. Helv. {\bf 48}(1973), 119--135.

\bibitem{rogo} W. Rogosinski, {\it On the coefficients of subordinate functions}, Proc. Lond. Math. Soc. {\bf 48(1)}(1945), 48--82.

\bibitem{peng} Z. Peng and G. Zhong, {\it Some properties for certain classes of univalent functions defined by differential inequalities}, Acta Math. Sci. Ser. B {\bf 37(1)}(2017), 69--78.

\end{thebibliography}
\end{document}